\documentclass[12pt,leqno,draft]{article}
\usepackage{bbm}
\usepackage{amsfonts}
\usepackage{stmaryrd}
\usepackage{amsmath, amsthm, amsfonts, amssymb, xcolor}
\usepackage{mathrsfs}
\newtheorem{thm}{Theorem}[section]
\newtheorem{cor}[thm]{Corollary}
\newtheorem{lem}[thm]{Lemma}
\newtheorem{prp}[thm]{Proposition}
\newtheorem{defprp}[thm]{Definition and Proposition}

\theoremstyle{definition}
\newtheorem{defn}[thm]{Definition}
\newtheorem{rem}[thm]{Remark}
\newtheorem{exa}[thm]{Example}
\newcommand{\scr}[1]{\mathscr #1}
\definecolor{wco}{rgb}{0.5,0.2,0.3}

\numberwithin{equation}{section} \theoremstyle{remark}

\newcommand{\ua}{\uparrow}

\title{\bf Compressibility and Stochastic Stability of Monotone Markov Chains} 
\author{Sergey Foss\footnote{School of MACS, Heriot-Watt University, EH14 4AS Edinburgh and Sobolev Institute of Mathematics. Email: s.foss@hw.ac.uk} and Michael Scheutzow\footnote{Technische Universit\"at Berlin, MA 7-5, Strasse des 17~Juni 136, 10623 Berlin. E-mail: ms@math.tu-berlin.de}}

\date{}

\begin{document}
\allowdisplaybreaks
\def\J{\mathcal J} \def\S{\mathcal S}    \def\R{\mathbb R}  \def\ff{\frac}  \def\B{\mathbf B}
\def\N{\mathbb N} \def\kk{\kappa} \def\m{{\bf m}}
\def\ee{\varepsilon}\def\ddd{D^*}
\def\dd{\delta} \def\DD{\Delta} \def\vv{\varepsilon} \def\rr{\rho}
\def\<{\langle} \def\>{\rangle} \def\GG{\Gamma} \def\gg{\gamma}
  \def\nn{\nabla} \def\pp{\partial} \def\E{\mathbb E}
\def\d{\text{\rm{d}}} \def\bb{\beta} \def\aa{\alpha} \def\D{\scr I}
  \def\si{\sigma} \def\ess{\text{\rm{ess}}}
\def\beg{\begin} \def\beq{\begin{equation}}  
\def\F{\scr F}
\def\e{\text{\rm{e}}} \def\ua{\underline a} \def\OO{\Omega}  \def\oo{\omega}
\def\P{\mathbf P} \def\ifn{I_n(f^{\bigotimes n})}
\def\C{\scr C}   \def\G{\mathcal G}   \def\aaa{\mathbf{r}}     \def\r{r} \def\CC{\mathcal{C}}
\def\gap{\text{\rm{gap}}} \def\prr{\pi_{{\bf m},\varrho}}  \def\r{\mathbf r}
\def\Z{\mathbb Z} \def\vrr{\varrho} \def\ll{\lambda}
\def\L{\scr L}\def\Tt{\tt} \def\TT{\tt}\def\II{\mathbb I}
\def\i{{\rm in}}\def\Sect{{\rm Sect}}  \def\H{\mathbb H}
\def\M{\scr M}\def\Q{\mathbb Q} \def\texto{\text{o}} \def\LL{\Lambda}
\def\Rank{{\rm Rank}} \def\B{\scr B} \def\i{{\rm i}} \def\HR{\hat{\R}^d}
\def\to{\rightarrow}\def\l{\ell}\def\iint{\int}
\def\EE{\scr E}\def\no{\nonumber}
\def\A{\scr A}\def\V{\mathbb V}\def\osc{{\rm osc}}
\def\BB{\mathbb B}\def\Ent{{\rm Ent}}\def\W{\mathbb W}
\def\U{\scr U}\def\8{\infty} \def\si{\sigma}
\def\33{\interleave}
\def\M{\mathcal{M}}
\def\K{\mathcal{K}}
\def\E{\mathcal{E}}
\def\L{\mathcal{L}}
\newcommand{\1}{\mathbf{1}}
\renewcommand{\bar}{\overline}
\renewcommand{\tilde}{\widetilde}
\maketitle

\begin{abstract}
For a stochastically monotone Markov chain taking values in a Polish space, we present a number of conditions for existence and for uniqueness of its stationary regime, as well as  for 
closeness
of its transient trajectories. In particular, we generalise a basic result by Bhattacharya and Majumdar (2007) where a certain form of mixing, or {\em splitting} condition was assumed uniformly over the state space. We do not rely on continuity properties of transition probabilities.
\end{abstract}



\section{Introduction}

The paper deals with various problems related to existence and uniqueness of the stationary distribution for stochastically monotone Markov chains. Namely, we consider a Markov chain on a Polish (complete and separable) metric space $(\E, d)$ and assume that it is equipped with a partial ordering that is compatible with the metric.
We assume further that the Markov transition probability kernel
preserves monotonicity (see Section 2 for further details). We introduce a concept of {\em compressibility} (which may be viewed as an analogue of the 
contractiveness property) and provide several sufficient conditions
for that to hold. Further, we provide sufficient conditions for the  existence and for the uniqueness
of the stationary distribution based on certain coupling constructions. We complement our results with a number of examples and counter-examples.

Our research is motivated by a stability result by Bhattacharya and Majumdar \cite{BM}  for a stochastically monotone Markov chain taking values in an interval of the real line  (see also Bhattacharya and Majumdar \cite{BM2} and  Bhattacharya and Waymire \cite{BW}, Corollary 19.8 and Theorem 19.9, for an extension to a closed subset of $\mathbb{R}^k$). 
 The authors assume that a certain mixing, or {\em splitting} condition holds uniformly over the whole state space.  The result was intensively used in a variety of applications. Motivated by models in economics, Foss, Shneer, Thomas and Worrall \cite{FSTW} obtained analogous results
for a class of monotone Markov-modulated Markov chains and also for certain non-Markovian models. Matias and Silva \cite{MS} extended results of \cite{BM} to a more general class of orderings in $\mathbb{R}^d$.
Butkovsky and Scheutzow \cite{BS}, \cite{BSC} considered continuous-time monotone Markov processes and provided criteria for the existence, uniqueness and weak (exponential) convergence to a stationary measure under a Lyapunov-type condition plus a {\em splitting} condition. Their splitting condition was similar to ours. Roughly speaking it says that for initial conditions $x \preceq y$ the chains starting from $x$ and $y$ can be coupled in such a way that after some finite random time the two chains are ordered in the opposite direction, so the order is {\em swapped}. 

There are a number of surveys on monotone Markov chains and monotone iterative maps, that quote and discuss the results of \cite{BM, BM2},  see e.g.~\cite{GM} and references therein.

The rest of our paper is structured as follows. In Section 2 we introduce basic definitions and notation, recall known results  related to monotonicity of Markov chains, and provide useful properties. Then, in Section 3, we discuss various conditions for compressibility of Markov chains. Section 4 deals with conditions for existence, and Section 5 with conditions for uniqueness of a stationary distribution. In Section 6 we summarise the obtained results and highlight their relation to those of Bhattacharya and Majumdar \cite{BM}, and provide further comments and examples. In Section 7 we comment on the representation of a Markov chain as a stochastic recursion 
and on relations between splitting conditions and stability.
Finally, in Section 8, we mention possible generalisations of our results.

\section{Couplings and monotone Markov chains}

\subsection{Couplings}


In this subsection, we recall and discuss several coupling constructions related to
probability measures, random variables and Markov chains.
Let $({\cal E}_1, \EE_1)$ and  $({\cal E}_2, \EE_2)$ be   measurable spaces. 

{\it Coupling of probability measures.}
For  probability measures $\mu$ on $({\cal E}_1, \EE_1)$ and  $\nu$ on $({\cal E}_2, \EE_2)$, a
coupling is a probability measure $\lambda$ on the product
space $({\cal E}_1\times {\cal E}_2, {\EE_1\otimes \EE_2})$ having $\mu$ and $\nu$ as
marginals, i.e. ${\mu}(A) = \lambda (A\times {\cal E}_2)$ and
$\nu (B) = \lambda ({\cal E}_1 \times B)$, for any $A\in {\EE}_1$ and $B\in {\EE}_2$.

{\it Coupling of random variables.} 
For two random variables, $\xi$ and $\eta$, that are defined, in general,  
on two different probability spaces and take values in   $({\cal E}_1, \EE_1)$ and  $({\cal E}_2, \EE_2)$, respectively,  a coupling is a pair $(\xi',\eta')$ of random variables on some
probability space $(\Omega, {\cal F}, {\bf P})$ such that $\xi =_d \xi'$ and $\eta =_d \eta'$.

Consider a time-homogeneous Markov chain (MC) $(X_n)_{n\in \N_0}$ on  a Polish space $({\cal E}, d)$ equipped with its Borel $\sigma$-algebra ${\cal \EE}$ (generated by the open sets in $\cal E$), and 
with transition kernel $P(x,A) = {\mathbf P} (X_{1}\in A \ | X_0=x)$, $x\in {\cal E}, A\in {\cal \EE}$.  We assume the usual measurability assumptions to be in force (for each $x$, $P(x,\cdot)$ is a probability measure and, for each $A$, $P(x , A)$ is measurable in $x$) but we do not assume the Feller property (i.e. continuity of the map $x \mapsto \int f(y) P(x,\d y)$ for each bounded and continuous function $f$).
For $n=1,2,\ldots$, denote by $P^{(n)}(x,A)$ the $n$-step
transition kernel.

From now on, we use the notation $(X_n^x)_{n \in \N_0}$ for the Markov chain
that starts from initial state $X_0^x=x\in \cal E$.

{\it Coupling of two Markov chains.}
One can couple two chains $(X_n^x)_{n \in \N_0}$ and $(X_n^y)_{n \in \N_0}$ with the same transition kernel and initial states $X_0^x=x$ and $X_0^y= y$ on a common probability space as a collection
of pairs $(X_n^x,X_n^y)_{n \in \N_0}$ of random variables. 
Note that we require each of the sequences $(X_n^x)_{n \in \N_0}$ and $(X_n^y)_{n \in \N_0}$ to form Markov chains, but their joint distribution   may be arbitrary. In particular, the sequence $(X_n^x, X_n^y)_{n \in \N_0}$ may or may not form a Markov chain 
on the product space ${\cal E} \times {\cal E}$. 
If it does, then the coupling is called a {\it bivariate Markovian} coupling (or {\it bi-Markovian} coupling, for short).  In this case, the coupling is characterized by  a Markov kernel $\bar P$ on ${\cal E} \times {\cal \E}$ with both marginals equal to $P$.
A standard example of a bi-Markovian  coupling is the  
{\it independent coupling} where, for any $A,B\in \EE$ and for any $n=0,1,\ldots$, 
\begin{align*}
{\mathbf P}((X_{n+1}^x, X_{n+1}^y)\in A\times B \, | 
\, (X_m^x,X_m^y)_{m\le n}) =
{\mathbf P} (X_{n+1}^x\in A \ | \ X_n^x) 
\cdot {\mathbf P} (X_{n+1}^y\in B \  | \  X_n^y) \ \ \text{ a.s.} 
\end{align*}
Another popular coupling, a {\it monotone bi-Markovian coupling},
will be considered in the next subsection.  

Sometimes, we consider couplings for a particular pair $(x,y)\in {\cal E}\times {\cal E}$ of initial conditions rather than for all pairs. In this case we call the corresponding coupling of the Markov chains starting in $x$ and $y$ an $(x,y)$-coupling. In the following we will mostly work with the concept of a marginally Markovian coupling which is a weaker concept than a bi-Markovian coupling.\\

\begin{defn}
  Let $x,y \in {\cal E}$ and let $(X_n^x,X_n^y)_{n \in \N_0}$ be an $(x,y)$-coupling of two Markov chains with transition kernel $P$. The coupling is called a {\em marginally Markovian $(x,y)$-coupling} if it satisfies
  \begin{align*}
\P\big(X^x_{n+1}\in \cdot\ |\ \sigma(X^x_k,X^y_k)_{k \leq n}\big)&=P\big(X_n^x,\cdot\big) \mbox{ and}\\
\P\big(X^y_{n+1}\in \cdot\ |\ \sigma(X^x_k,X^y_k)_{k \leq n}\big)&=P\big(X_n^y,\cdot\big)
  \end{align*}
  almost surely for any $n \in \N_0$.
\end{defn}

Clearly, every bi-Markovian coupling is marginally Markovian.

We will now consider couplings for more than one pair of initial conditions. To do this, we use two upper indices for the chain starting in $x$, the first is $x$ and the second one refers to the process starting in $y$ and which we want to couple with the first process.

\begin{prp}\label{prp3.2}
  Fix $x,y\in {\cal E}$ and suppose that $(X_n^x,X^y_n)_{n \in \N_0}$ is a marginally Markovian $(x,y)$-coupling. Assume that, for each $\hat x,\hat y \in {\cal E}$, $(\hat X_n^{\hat x,\hat y},\hat X^{\hat y,\hat x}_n)_{n \in \N_0}$ is a marginally Markovian $(\hat x,\hat y)$-coupling, all defined on a common probability space with    $(X_n^x,X^y_n)_{n \in \N_0}$ independent of all $\hat X$-variables (whatever their upper and lower indices).
  Let $\tau$ be a stopping time with respect to $\sigma\big(X_k^x,X_k^y\big)_{k \leq n}$, $n \in \N_0$. Then the sequence
  $$
 (\tilde X_n, \tilde Y_n):=\left\{
    \begin{array}{ll}
(X_n^x,X_n^y)& \mbox{ if }n \leq \tau\\
      (\hat X^{u,v}_{n-\tau} ,\hat X^{v,u}_{n-\tau})& \mbox{ if }n > \tau \mbox{ and } (X_\tau^x,X_\tau^y)=(u,v)
    \end{array}
    \right.
    $$
 is a marginally Markovian $(x,y)$-coupling.   
\end{prp}

In what follows, we call the marginally Markovian $(x,y)$-coupling $(\tilde X_n,\tilde Y_n)_{n\ge 0}$
a {\it concatenation} of the two couplings.

We will provide a straightforward proof of the proposition in an appendix.

\begin{defn}
Let a set $H \in \EE \otimes \EE$ be given. 
Assume that, for any pair $(x,y)\in {\cal E}\times {\cal E}$,  there
is a marginally Markovian $(x,y)$-coupling  $(X_n^x,X_n^y)_{n\ge 0}$ such that the hitting time of 
the set $H$,
\begin{align*}
\tau_{x,y} \equiv \tau_{x,y}(H)
=
\min \{n\ge 0: (X_n^x,X_n^y)\in H\}
\end{align*}
is almost surely finite.
Then the set $H$ is {\it achievable} and the corresponding family of marginal Markovian couplings is called {\it $H$-successful}, or
just {\it successful}.
 \end{defn}

\begin{rem}
The independent and the monotone (see the next subsection)  couplings are the most popular couplings in the literature. 
However, in some cases, there may exist a ``better'' coupling that makes 
 a pair of Markov chains recurrent while
the independent 
coupling is transient.
We provide an example of a random walk with a discrete  distribution of jumps. A similar example of a random walk in the plane for which the increments have a continuous distribution may be easily constructed, too.

\begin{exa} 
Let $(X_n)_{n \in \N_0}$ be a simple random walk on the 2-dimensional lattice $\Z^2$,  $X_n^x = x+ \sum_1^n \eta_i$ where $\eta_n=(\eta_{n,1}, \eta_{n,2}),\,n \in \N$ is an i.i.d.~sequence of 2-dimensional vectors with independent coordinates, with each of them taking three values,
$-1, 0$ and $1$ with equal probabilities $1/3$. This random walk forms an aperiodic null-recurrent Markov chain. 

For $x\neq y$, an independent coupling of $(X_n^x,X_n^y)$ is a  zero-mean 4-dimensional random walk which is known to be transient. 
On the other hand, we may use the following coupling that is a concatenation of independent and of ``identical'' couplings. For any $x=x_0\equiv (x_{0,1},x_{0,2}), y=y_0=(y_{0,1},y_{0,2})$, any $n$ and
$x_k=(x_{k,1},x_{k,2}), y_k=(y_{k,1},y_{k,2})$, $k=n,n+1$, 
we let
\begin{align*}
&{\mathbf P} (X_{n+1}^x =x_{n+1}, X_{n+1}^y = y_{n+1} \ | \
X_{n}^x =x_{n}, X_{n}^y = y_{n})\\
= \ &{\mathbf P} (X_{n+1}^x=x_{n+1} \ | \ X_n^x = x_n) \cdot
 {\mathbf P} (X_{n+1}^x=y_{n+1} \ | \ X_n^x = y_n)
\end{align*}
if $x_n\neq y_n$ (so the jumps are independent) and

\begin{align*}
&{\mathbf P} (X_{n+1}^x =x_{n+1}, X_{n+1}^y = y_{n+1} \ | \
X_{n}^x =x_{n}, X_{n}^y = y_{n}) \\
= \ & {\mathbf P} (X_{n+1}^x=x_{n+1} \ | \ X_n^x = x_n) \cdot
\1_{x_{n+1}=y_{n+1}}
\end{align*}
if $x_n = y_n$ (so the jumps are identical).\\ 
This means that  we proceed with independent coupling $(X_n^x,X_n^y)$ (where $X_n^x=(X_{n,1}^x,X_{n,2}^x)$
and $X_n^y=(X_{n,1}^y,X_{n,2}^y)$) until time $\nu_1=\nu_1(x,y)$ which is
the first time when $X_{n,1}^x=X_{n,1}^y$ and $X_{n,2}^x=X_{n,2}^y$, or, equivalently, the first time when $X_n^x-X_n^y=(0,0)$.
Since a sequence $Y_n = X_n^x-X_n^y$ is a 2-dimensional zero-mean random walk with independent coordinates having a symmetric distribution on the set $\{-2,-1,0,1,2\}$, such a random walk is recurrent and, therefore, $\nu_1(x,y)$ is 
finite a.s., for any $x$ and $y$ from the 2-dimensional lattice.\\
Then, starting from time $\nu_1$, we continue with the identical coupling, so $X_n^x=X_n^y$ for all $n\ge \nu_1$.
Since $X_n^x$ is a 2-dimensional simple zero-mean random walk which is recurrent, there exists an a.s. finite time 
$\nu_2 \ge \nu_1$ when $X_n^x=X_n^y=(0,0)$, for all $x$ and $y$ (or, equivalently, this coupling is $H$-successful, with $H = \{(0,0)\}\times \{(0,0)\}$). 
Note that this is a bi-Markovian coupling.
\end{exa}

\end{rem}

\subsection{Stochastic monotonicity}

Until the end of the paper, we deal with a Polish state space $({\cal E},d)$ equipped with the $\sigma$-algebra $\EE$ generated by the open sets. We assume further that the state space is equipped with a {\it partial ordering} $\preceq$ that is compatible with the topology generated by the metric $d$ (see e.g. \cite{Li} or \cite{FGS}), namely, that 
\begin{align}\label{compatibility}
\text{the set} \ \ M:=\{(x,y):\,x \preceq y\} \ \ 
 \text{is closed in}  \ \ {\cal E} \times {\cal E}.
\end{align}
We call such a space an {\em ordered Polish space}.

We say that a set $I \in \EE$ is {\it increasing} if, for any 
 $x\in I$ and $y\in {\cal E}$, the  inequality $x \preceq y$ implies $y\in I$. We denote the set of all measurable and increasing sets by $\D$.

A real-valued measurable function $g: {\cal E} \rightarrow {\mathbb{R}}$ is called {\it monotone} or {\it increasing}
if $x \preceq y$ implies $g(x)\le g(y)$. 
We denote  by ${\mathcal G}$ the class of all bounded and monotone measurable  functions taking values in $[0,1]$ (restricting the range to the unit interval is convenient  and does not lead to any loss of generality).

We denote the set of all probability measures on $(\E, d)$ by ${\cal{M}}_1(\E)$.  For $\mu_1,\mu_2\in {\cal{M}}_1(\E)$ we say that $\mu_1$ is {\it stochastically smaller than} $\mu_2$
and write $\mu_1 \preceq \mu_2$ if $\mu_1(I) \le \mu_2(I)$,
for any set $I\in \D$. Note that  $\mu_1 \preceq \mu_2$ implies  $\int g(x)\,\d \mu_1(x)\leq \int g(x)\,\d \mu_2(x)$ for every $g \in {\mathcal G}$ which is easy to see by considering non-negative finite linear combinations of indicator functions and their limits.

We say that a Markov chain with transition kernel $P$ is {\it stochastically monotone} ({\it monotone}, for short) if 
the mapping ${\cal G} \ni g \mapsto Pg$ preserves monotonicity, i.e.~$g \in \cal G$ implies that  the function $Pg(x) := \int_{\cal E}
g(y) P(x, \d y)$ belongs to $\cal G$ as well. 
Note that it is enough to ensure monotonicity by considering indicator functions $g=\1_I$, $I \in  \D$  only.
Denoting $\mu P(\cdot):=\int P(x,\cdot)\,\d \mu(x)$ for a probability measure $\mu$, it is clear that $P$ is monotone iff for every pair of probability measures
$\mu_1 \preceq \mu_2$, it follows that $\mu_1P \preceq \mu_2 P$.

\begin{prp}\label{prop1} (Strassen's theorem, see \cite{KKO} or  \cite{Li}).
  Assume that $\cal E$ is an ordered Polish space.
If $\mu_1,\,\mu_2 \in {\cal M}_1(\cal E)$ satisfy $\mu_1 \preceq \mu_2$ then there is a coupling $\lambda$ of $\mu_1$ and $\mu_2$ for which $\lambda(M)=1$.
\end{prp}

\begin{prp}\label{prosi} 
(\cite{KKO}, Proposition 4).  
Assume that the ${\cal E}$-valued Markov chain $(X_n)_{n\in \N_0}$ is monotone. 
For any finite or countable sequence $x_1\preceq x_2 \preceq \ldots \preceq x_k \preceq \ldots$ of elements of $\cal E$, there exists a coupling 
of the Markov chains $(X_n^{x_1})_{n \in \N_0}, (X_n^{x_2})_{n \in \N_0},\ldots, (X_n^{x_k})_{n \in \N_0},\ldots$      
such that
$X_n^{x_1}\preceq X_n^{x_2} \preceq \ldots \preceq X_n^{x_k}\preceq \ldots$ a.s., for all $n$.
\end{prp}

We will need the following facts. The first one is a special case of \cite[Lemma 1]{KK} while the second one is essentially  \cite[Theorem 2]{KK}.

\begin{prp}\label{cons}
Let $\pi_1$ and $\pi_2$ be two probability measures on ${\cal E}$ which agree on  $\D$. Then $\pi_1=\pi_2$.
\end{prp}

\begin{prp}\label{orderedPolish}
  The space ${\cal M}_1(\cal E)$ of probability measures equipped with a complete  metric generating weak convergence and the stochastic  order $\preceq$ is an ordered Polish space.
\end{prp}

\begin{proof} The fact that for each Polish space $\cal E$ the space ${\cal M}_1(\cal E)$ is itself Polish with respect to some complete metric which generates the topology of weak convergence is well-known. Antisymmetry of the partial order follows from Proposition \ref{cons} and the fact that this space satisfies property \eqref{compatibility} is the content of \cite[Proposition 3]{KKO}.
\end{proof}

\begin{prp}\label{KKO}
If $\mu,\mu_1,\mu_2,...\in {\cal M}_1({\cal E})$ satisfy $\mu_1 \preceq \mu_2 \preceq ...$ and $\mu_i \preceq \mu,\;i \in \N$, then there exist ${\cal E}$-valued random variables $X,X_1,X_2,...$ on some probability space  with respective laws $\mu,\mu_1,\mu_2,...$ such that $X_1 \preceq X_2 \preceq ...\preceq X$ almost surely. 
\end{prp}

\begin{proof} Due to \cite[Proposition 4]{KKO} there exist random variables $X_1 \preceq X_2 \preceq ...$ on some space with respective laws $\mu_1,\mu_2,...$. Define
  $F:={\cal E}^\N$ equipped with the product topology and componentwise ordering. Then $F$ is again an ordered Polish space (see \cite[p.901]{KKO}). Denote the joint law of $(X_1,X_2,...)$ on $F$ by $\nu$. Let $\bar \nu$ be the law of $Y,Y,...$ where $Y$ has law $\mu$. Note that, for each $i \in \N$, the law of $(X_1,...,X_i)$ is dominated by the law of $(Y,...,Y)$ on the space of probability measures ${\cal M}_1({\cal E}^i)$. Hence, by \cite[Proposition 2]{KKO}, we have $\nu \preceq \bar \nu$. Applying  Strassen's theorem yields the conclusion of the proposition.
\end{proof}

We introduce the following condition:\\

{\bf (BMC, bounded-monotone-convergence)}
If an increasing sequence
$ x_1 \preceq x_2 \preceq \ldots \preceq x_n \ldots$ of elements from $\E$ is bounded  above by $y\in \E$ (i.e. $ x_n\preceq y $ for all  $n$),  then the
sequence $(x_n)_{n \in \N}$ converges to some $x\in \E$.\\

Note that condition (BMC) and condition \eqref{compatibility} imply that 
$x\preceq y$ and $x_n\preceq x$, for all $n$.

\begin{exa}
  An example where condition (BMC) fails is the space 
  $C([0,1],\mathbb{R})$ of real-valued continuous functions on $[0,1]$ equipped with the sup-norm and order $f \preceq g$ if $f(x)\leq g(x)$ for all $x \in [0,1]$. Indeed, a pointwise limit of an increasing sequence of continuous functions may have jumps. Another example is stated in Remark \ref{remex}.
\end{exa}

\begin{rem}
  A sufficient (but not necessary) condition for (BMC) to hold is that for each $a \preceq b$,  the interval $[a,b]:=\{x:a \preceq x \preceq b\}$ is compact. Indeed, if $a_1 \preceq a_2 \preceq \ldots \preceq b$, then, by assumption, $[a_1,b]$ is compact and therefore each subsequence has a further subsequence which converges. If $\bar a$ and $\tilde a$ are two limit points of such subsequences, then  $\bar a \preceq \tilde a$ and
  $\tilde a \preceq \bar a$, so $\bar a=\tilde a$. Therefore, the whole sequence converges and (BMC) holds. 
\end{rem}
  
The following is an immediate consequence of Proposition \ref{KKO}. 
\begin{cor}\label{BMCmeasures}
  Assume that (BMC) holds for $\cal E$. Then a similar condition holds for $\M_1(\cal E)$: 
 if $\mu_1 \preceq \mu_2 \preceq \ldots$ is an increasing sequence of probability measures which is bounded by the probability measure $\lambda$, then there exists a measure $\mu$ to which the sequence converges weakly.
  \end{cor}
  
  We will generally not assume condition (BMC) to hold. Whenever we want it to hold, we will explicitly say so.

The following auxiliary result will be used later. 

  \begin{lem}\label{neww}
Let $\mu_n$, $n \in \N$ be a sequence of probability measures on $\cal E$ which converges weakly to $\tilde \pi$ and assume that $\pi$ is a probability measure such that $\lim_{n \to \infty}\mu_n(I)=\pi(I)$ for every $I \in \D$. Then $\tilde \pi=\pi$.  
\end{lem}

\begin{proof}
  If $I \in \D$ is closed, then
  $$
\tilde \pi(I)\geq \limsup_{n \to \infty} \mu_n(I)=\pi(I),
$$
so \cite[Theorem 1]{KKO} implies $\tilde \pi \succeq \pi$. Similarly, if $I \in \D$ is open, then  $\tilde \pi(I)\leq \pi(I)$ and, since complements of increasing
open sets are decreasing closed sets, we can (by symmetry) again invoke  \cite[Theorem 1]{KKO} to obtain $\tilde \pi \preceq \pi$, so  $\tilde \pi=\pi$.
\end{proof}

We conclude this section by introducing (bivariate) monotone couplings. Since monotone couplings which are not bivariate will not be relevant for us, we will make the bivariate property part of the definition.  

  \begin{defprp} (see \cite[Theorem 2.3.]{MaSh}).
    Let $P$ be the Markov kernel of a monotone Markov chain. There exists a Markov kernel $\bar P$ on ${\cal{E}}\times{\cal{E}}$ with both marginals equal to $P$ with the property that $\bar P\big((x,y),M\big)=1$ whenever $x \preceq y$. The corresponding bi-Markovian coupling is called {\em monotone bi-Markovian coupling} or simply {\em monotone coupling}.
  \end{defprp}

\section{Compressibility and its sufficient conditions}

As before, $({\cal E},\preceq)$ is an ordered Polish space with  Borel $\sigma$-algebra $\EE$ and $P$ is a monotone Markov kernel on ${\cal E}$. Recall that 
$M=\{(x,y): x \preceq y\}$ is closed. For subsets $A$ and $B$ of ${\cal E}$, we write $A \preceq B$ if $x \preceq y$ for every $x \in A$ and $y \in B$. For $x \in {\cal E}$ and $B \subseteq {\cal E}$, we write $x \preceq B$ if $x \preceq y$ for any $y \in B$.

In part (iii) of the following theorem we will need an additional assumption which we call {\em KS-normal} since it is similar (but not equivalent) to the concept of {\em normality} of an ordered Polish space
(see e.g.~\cite{FGS}) and it was introduced by Kamihigashi and Stachurski in \cite{KS}. For a subset $K$ of $E$, we define $i(K):=\{x\in E:\, x\succeq y \mbox{ for some }y \in K\}$ and $d(K):=\{x\in E:\, x\preceq y \mbox{ for some }y \in K\}$. If $K$ is compact then it is easy to see (and well-known) that $i(K)$ and $d(K)$ are closed sets.\\

\begin{defn} (\cite[Assumption 4.1]{KS}) The ordered Polish space $E$ is called {\em KS-normal} if for every compact set $C \subseteq E$, the set $i(C)\cap d(C)$ is also compact.
\end{defn}

Recall that a family $\cal M$ of probability measures on a Polish space $\cal E$ is called {\em tight} if for every $\varepsilon>0$ there exists a compact set $K$ in $\cal E$ such that $\mu(K)\geq 1-\varepsilon$ for every $\mu \in \cal M$. 

\subsection{Generic results}

\begin{thm}\label{L1}
  For a given transition kernel $P$, assume that the set $M$ is achievable. Then,
  \begin{itemize}
\item[i)]  
for any $x,y \in {\cal E}$ and as $n\to\infty$,
\begin{align}\label{conv}
\sup_{g\in {\cal G}} \big| {\mathbf E} g(X_n^x) - {\mathbf E} g(X_n^y)\big| \leq \max \Big\{ {\mathbf P}\big( \tau_{x,y}(M)> n\big), {\mathbf P}\big( \tau_{y,x}(M)> n\big)\Big\}\to 0. 
\end{align}
\end{itemize}
If we additionally assume that there exists a stationary distribution $\pi$ for the Markov chain, then $\pi$ is unique and the following hold:
\begin{itemize}
\item[ii)] For each $x \in {\cal E}$, 
\begin{align}\label{convint}
\lim_{n \to\infty}\sup_{g\in {\cal G}} \Big| {\mathbf E} g(X_n^x) - \int g(y)\,\d \pi(y)\Big| \le 
\lim_{n \to\infty}
\sup_{g \in \G} \int\big|{\mathbf E} g(X_n^x)- {\mathbf E} g(X_n^y)\big|\,\d \pi (y)
=0. 
\end{align}
\item[iii)] 
If, in addition, ${\cal E}$ is KS-normal, then, for every $y \in {\cal E}$, 
the distribution of $X_n^y$ converges to $\pi$ weakly as $n \to \infty$.
\item[iv)] If, for some $x \in {\cal E}$, the laws of  $X_n^x,\;n \in \N_0$ are tight, then they converge to $\pi$ weakly (even if ${\cal E}$ is not KS-normal).  
\end{itemize}
\end{thm}

\begin{rem}
In what follows, we call property \eqref{conv} a {\it compressibility
property}.
\end{rem}

\begin{proof}[Proof of Theorem \ref{L1}]
The proposed proof uses the concatenation of 
marginally Markovian couplings.

Proof of i). Take any pair $x\neq y$. Consider an $M$-successful 
marginally Markovian $(x,y)$-coupling of $\big(X_n^x\big)$ and $\big(X_n^y\big)$.
Then $\tau\equiv \tau_{x,y}(M) <\infty$ a.s. 
We concatenate it with an independent  monotone 
bi-Markovian
coupling. 
By Proposition  \ref{prp3.2} the concatenation  is a 
marginally Markovian $(x,y)$-coupling   which, for ease of notation, we denote by $(X^x_n,X^y_n)_{n \in \N_0}$ (skipping the tilde on top of $X$).
    
For any function $g\in \cal G$,
\begin{align*}
{\mathbf E} (g(X_n^x))
&= {\mathbf E} \big( g (X_n^x)\1_{\{n< \tau\}}\big) +
{\mathbf E}\big( g(X_n^x)\1_{\{n\ge \tau\}}\big) \\
&\le {\mathbf P} (n < \tau) +
{\mathbf E} \big( g(X_n^y)\1_{\{n\ge \tau\}}\big)\\
&\le {\mathbf P}(n< \tau) + {\mathbf E}(g(X_n^y)).
\end{align*}
Thus, 
\begin{align*}
{\mathbf E} (g(X_n^x)) - {\mathbf E} (g(X_n^y))\le 
{\mathbf P}(\tau_{x,y}(M)>n).
\end{align*}
Now we may swap $x$ and $y$ in the previous construction, to obtain 
\begin{align*}
{\mathbf E} (g(X_n^y)) -  {\mathbf E}(g(X_n^x))
\le {\mathbf P}(\tau_{y,x}(M)>n).
\end{align*}
 Therefore,
\begin{align*}
|{\mathbf E}(g(X_n^y)) - {\mathbf E}(g(X_n^x))|
\le \max \Big\{{\mathbf P}(\tau_{x,y}(M)>n), 
{\mathbf P}(\tau_{y,x}(M)>n)\Big\}.  
\end{align*}
Since the right-hand side of the above inequality does not depend on the function $g$ and since both random variables $\tau_{x,y}$ and $\tau_{y,x}$ are proper, we conclude that 
\begin{align}\label{conv11}
\sup_{g\in {\cal G}}\ |{\mathbf E}(g(X_n^y)) - {\mathbf E}(g(X_n^x))| \to 0
\quad \text{as} \quad n\to\infty.
\end{align}
This proves \eqref{conv}.\\

Proof of ii). Fix $x \in \E$ and a stationary distribution $\pi$. Then,
\begin{align*}
\sup_{g \in \G} \Big|{\mathbf E} g(X_n^x)-\int g(y)\,\d \pi (y)     \Big|
&=
\sup_{g \in \G} \Big|{\mathbf E} g(X_n^x)-
\int {\mathbf E} g(X_0^y)\,\d \pi (y)     \Big|\\
&=\sup_{g \in \G} \Big|{\mathbf E} g(X_n^x)-
\int {\mathbf E} g(X_n^y)\,\d \pi (y)     \Big|\\
&=\sup_{g \in \G} \Big|\int \left( {\mathbf E} g(X_n^x)- 
{\mathbf E} g(X_n^y)\right) \,\d \pi (y)     \Big|\\
&\leq \sup_{g \in \G} \int\big|{\mathbf E} g(X_n^x)- {\mathbf E} g(X_n^y)\big|\,\d \pi (y)\to 0, 
\end{align*}
as $n\to\infty$, by \eqref{conv} and dominated convergence. Proposition \ref{cons} implies that $\pi$ is unique. \\

iii) 
Let $\varepsilon>0$ and choose a compact set $C$ such that $\pi(C)>1-\varepsilon$. By assumption, the set $\bar C:=i(C) \cap d(C)$ is compact. Since $i(C)$ and $i(C)\backslash d(C)$ are in $\D$ and $d(C)^c$ and $\big(d(C)\backslash i(C)\big)^c$ are in $\D$ it follows from  part ii) that for each $y \in E$ the sequence $\P\big(X_n^y\in \bar C\big)$ converges to
    $\pi(\bar C)$ which is larger than $1-\varepsilon$. Since $\varepsilon>0$ is arbitrary it follows that the laws of
    $X_n^y$, $n \in \N_0$ are tight and the claim follows from part iv).
    \\

iv)  Tightness implies that every subsequence of $\L \big(X_n^x\big),\;n \in \N_0$, has a further subsequence which converges weakly to some probability measure $\tilde \pi$ 
(this follows from Prokhorov's theorem). By (ii) and Lemma \ref{neww},  we have
$\tilde \pi=\pi$, so $\tilde \pi$ is independent of the choice of the subsequence and the claim follows.
\end{proof}

\begin{rem}
Tightness of the laws of  $X_n^x,\;n \in \N_0$ does not imply the existence of a stationary distribution $\pi$ 
(not even if ${\cal E}$ is KS-normal):

Consider the space $\E=\{0,1,1/2,1/3,1/4,...\}$ equipped with the Euclidean metric and the trivial order (i.e.~no pair of distinct states is comparable). This defines a KS-normal ordered Polish space. Consider the chain on $\cal{E}$ which remains in the current state with probability $1/2$ and moves to the next state (in the order given above) with probability $1/2$. The chain is trivially monotone, $M= \{(x,x), x\in\E\}$ is achievable and all transition probabilities converge weakly to a Dirac measure on $0$. In particular, all transition probabilities are tight but, obviously, there is no stationary measure.
\end{rem}

\begin{rem}
  The conditions of Theorem \ref{L1} plus existence of a stationary distribution  $\pi$ do {\em not} imply weak convergence of all transition probabilities to $\pi$ as the following example shows.
  
Equip $\E=\N_0 \times [-1,1]$ with the metric induced by the Euclidean metric in $\mathbb{R}^2$ and consider the chain with transitions 
\begin{align*}
(0,x) &\mapsto (0, \frac x2 +\cal{U})\\
(j,x) &\mapsto (j+1, \frac x2+{\cal{U}}),  j \geq 1, 
\end{align*}
for $x \in [-1,1]$ and $ \cal{U}$ uniformly distributed on $\{-\frac 12,\frac 12\}$. 

Obviously, the chain has a unique stationary distribution $\pi$ (concentrated on the states with first component 0; the stationary distribution is in fact the uniform distribution on those states) and no sequence of transition probabilities with first component larger than 0 converges to $ \pi$ weakly. 

Our aim is to equip $\E$ with a partial order which makes $(\E,\preceq )$ an ordered Polish space for which the chain is monotone and $M$ is achievable. 

Define
\begin{align*}
(i,x)\prec (i,y)& \mbox{ iff } x<y\\ 
(i,x)\prec (j,y) &\mbox{ iff } j >i \mbox { and } y \geq x + h_{i,j}\\
(j,y) \prec (i,x) &\mbox{ iff } j >i \mbox { and } y \leq x-h_{i,j}
\end{align*}
where $h_{i,j}$ are strictly positive numbers which we specify now. We choose $h_{0,j}=2^{-j}$ and  $h_{i,j}=\sum_{k=i+1}^j 2^{-k}$ in case $i \geq 1$. This defines a partial order and  $(\E,\preceq )$ is an ordered Polish space. To see that the chain is monotone, we  couple all one-step transitions by  choosing the same realization of the random variable $\cal{U}$ for all initial conditions.  Then,  the states after one step are almost surely ordered for any given pair of initial conditions which are ordered. 
Finally, let $A$ be the set of all pairs in $\E$ for which the second coordinate is at most $-\frac 34$ and $B$ the set for which the second coordinate is at least $\frac 34$. Clearly, we have $A \prec B$ and the set 
$A \times B$ and hence $M$ are achievable. 

\end{rem}

\subsection{Sufficient conditions for achievability of the set $M$}

In this subsection, we present two sets of sufficient conditions
for achievability of the set $M$. We start with conditions formulated in terms of coupled Markov chains.

\begin{lem}\label{L2-1}
Assume that there exists a set $C\in {\EE}$ such that 
\begin{align}\label{condR}
\text{the set} \ \ C\times C \ \  \text{is achievable};
\end{align}
and that there exist 
two measurable sets $A \preceq B$
and numbers $\varepsilon >0$ and 
$N\ge 1$ such that   
 \begin{align}\label{condS}
P^{(N)}(x,A) \geq \varepsilon \ \ \text{and} \ \  
 P^{(N)}(x, B) \geq \varepsilon, \ \ \text{for any}  \ \ x\in C.
\end{align}
Then
the set $M$ is achievable.
\end{lem}
Here, condition \eqref{condR} is formulated in terms of a coupling of two Markov chains, and it may not be easy to verify it. Below we provide sufficient conditions for the conclusion of  Lemma \ref{L2-1} to hold, that are formulated in terms of a ``marginal'' Markov chain.

\begin{lem}\label{L2-verynew}
Assume there is an integer $N\ge 1$ such 
that 
condition \eqref{condS} holds with the following additional
constraint:
\begin{align}\label{condSS2verynew}
\text{there exists}  \ \ z_0\in C \ \  \text{such that} \ \ 
A \preceq z_0 \preceq B;
\end{align}
and that 
\begin{align}\label{tauxC}
\tau_x(C) := \inf \{n\ge 1: X_{Nn}^x \in C\} < \infty \ \ \text{a.s., for all} \ \ x\in {\cal E}
\end{align}
and, moreover, there exists a positive and increasing function $g:\N \to \mathbb{R}$
such that
\begin{align}\label{functiong}
\sum_{n=1}^\infty 1/g(n) <\infty
\end{align}
and 
\begin{align}\label{supEg}
\sup_{x\in C} {\mathbf E} g\left(\tau_x(C)\right) < \infty.
\end{align}
Then the set $A\times B$ is achievable, and therefore the set $M$ is achievable, too.

\end{lem}

\begin{rem}\label{remg}
Examples of functions $g$ satisfying \eqref{functiong} include
$g(n)= n^{\alpha}$ and $g(n) = n (\log (n+1))^{\alpha}$, for some $\alpha >1$.
\end{rem}
    
\begin{rem}\label{rem38}
One can get upper bounds on the expected time to achieve the
set $A\times B$ starting from states $x$ and $y$. For example,
if ${\mathbf E} (\tau_x (C))^{\alpha}$, 
${\mathbf E} (\tau_y(C))^{\alpha}$ and 
$\sup_{z\in C} {\mathbf E} (\tau_z(C))^{\alpha}$ are all finite for
some $\alpha >1$,
then there exists a coupling such that the hitting time of the set $A\times B$ has a finite $\alpha$'s moment, too.
See e.g the proof of Theorem 7.3 in Chapter 10 of \cite{Tho} for
a similar consideration.
Similarly, finiteness of exponential moments of the former three terms imply finiteness of an exponential moment of the hitting time. 
\end{rem}

\begin{proof}[Proof of Lemma \ref{L2-1}]

The proposed proof is based on a particular coupling construction
that combines iteratively two types of 
 couplings.

Take any  $x\neq y$. We construct a coupling of $(X_n^x)$ and $(X_n^y)$ and  a random  time
$U<\infty$ a.s. such that $X_U^x\preceq X_U^y$ a.s. 
Then it will follow that $\tau_{x,y}(M)\leq U$ is a.s.~finite too.

Step 1. We start with a $C\times C$-successful coupling of $\big(X_n^x\big)_{n \in \N_0}$ and $\big(X_n^y\big)_{n \in \N_0}$
until time $T_1=\tau_{x,y}(C\times C)$. Then we concatenate it with a  coupling $\big(\hat X_n^{\hat x},\hat X_n^{\hat y}\big)_{n \in \N_0}$ with independent components as in
Proposition \ref{prp3.2}. We continue to use the letter $X$ instead of $\tilde X$ for the concatenated coupling. \\

If the event
\begin{align}\label{T1}
E_1=\{X_{T_1+N}^x \in A, X_{T_1+N}^y \in B\} 
\end{align}
occurs, then, clearly, $ X_{T_1+N}^x \preceq X_{T_1+N}^y$, and we stop the procedure, letting $U=T_1+N$.
Otherwise, we proceed with \\
Step 2:
starting at time $T_1+N$ from values $\widehat{x}=
X_{T_1+N}^x$ and $\widehat{y}=X_{T_1+N}^y$, we run another
$C\times C$-successful coupling until time $T_1+N+T_2$,
\begin{align*}
T_2 = \min \{n > T_1+N: (X_n^x,X_n^y)\in C\times C\}
\end{align*}
and then another independent coupling 
at times $n=T_1+N+T_2+1, \ldots,T_1+N+T_2+N$.\\
If the event 
\begin{align}\label{T2}
E_2=\{ X_{T_1+N+T_2+N}^x \in A,  X_{T_2+N-1}^y \in B\}
\end{align}
occurs, we stop and let $U=T_1+N+T_2+N$, otherwise continue
with\\ 
Step 3: another $C\times C$-successful coupling, followed by an independent coupling of length $N$. An induction argument then completes the construction.

Note that the probability of the event $E_1$ is not smaller than
$\varepsilon^2$. Further, given that $E_1$ does not occur, the probability of $E_2$ is again not smaller than $\varepsilon^2$,
and so on. Therefore, we stop after a random number $\nu$
of steps where
$\nu$ is bounded above by a random variable having a geometric distribution with parameter $\varepsilon^2$ and, therefore, is finite a.s. Thus, 
\begin{align*}
U = \sum_{i=1}^{\nu} T_i + \nu N
\end{align*}
 is a.s. finite too.

\end{proof}

\begin{proof}[Proof of Lemma \ref{L2-verynew}]

We start with two observations.\\

Observation 1: Without loss of generality, we may and will assume that the sets $A$ and $B$ are of the form
\begin{align}\label{ABnew}
A = \{x\in {\cal E}: x\preceq z_0\} \ \ \text{ and } 
\ \ 
B = \{y\in {\cal E}: z_0\preceq y\}.
\end{align}
Then, in particular, $A$ and $B$ are closed and $A$ is decreasing and $B$ is increasing.\\

Observation 2: By  stochastic monotonicity of the transition kernel, ${\mathbf P} (X_N^x \in A) \ge \varepsilon$, for any $x\in A$ and ${\mathbf P} (X_N^y \in B) \ge \varepsilon$, for any $y\in B$.
In particular, these properties hold for $x=z_0$ and for $y=z_0$ and monotonicity and the fact that $A$ is decreasing imply that for any $x \in A$ we have ${\mathbf P} (X_N^x \in A) \ge {\mathbf P} (X_N^{z_0} \in A) \ge \varepsilon$ and, since $B$ is increasing, ${\mathbf P} (X_N^y \in B) \ge {\mathbf P} (X_N^{z_0} \in B) \ge \varepsilon$ for any $y \in B$.\\

Now we turn to the main part of the proof. We first assume that $N=1$.\\

For $x\in {\cal E}$, write for short 
$\tau_x = \tau_x(C)$. 
For  $k\ge 1$, let
$\chi_k^x = \min \{n\ge k: X_n^x\in C\}-k$.
Further, let $R:= \sup_{z\in C} {\mathbf E} g(\tau_z)$, which is finite by the assumptions in the lemma.

Clearly, for any $m\ge 0$, 
\begin{align*}
  {\mathbf P} (\chi_k^x >m) &= {\mathbf P} (\tau_x >k+m) + \sum_{i=0}^{k-1}{\mathbf P} \big( X^x_{i+1}\in C,\,   \chi_i^x >k-i+m-1 \big)\\
                            &={\mathbf P} (\tau_x >k+m) + \sum_{i=0}^{k-1}{\mathbf P} \big( \chi_{i+1}^x >k-i+m-1\big| X^x_i\in C \big)  {\mathbf P} \big(    X^x_i\in C    \big)\\
                            &\leq{\mathbf P} (\tau_x >k+m) + \sum_{i=0}^{k-1}{\mathbf P} \big( \chi_{i+1}^x \geq k-i+m  \big| X^x_i\in C \big)\\
                            &\leq {\mathbf P} (\tau_x >k+m) + \sum_{i=0}^{k-1} \sup_{z \in C} {\mathbf P} (\tau_z \geq k+m-i)\\
                            &\leq {\mathbf P} (\tau_x >k+m) + \sum_{i=0}^{k-1} \frac R{g(k+m-i)}.\\
  &\leq  {\mathbf P} (\tau_x >k+m) +R \sum_{j=0}^{\infty}1/g(j+m+1), 
\end{align*}
where we applied Markov's inequality in the penultimate step. By condition \eqref{functiong}, we may choose $k_0=k_0(x)$ and $m_0$ (not depending on $x$!) such that, for $k \geq k_0$, we have
\begin{equation}\label{choose}
  {\mathbf P} (\chi_{k}^x >m_0)\leq 1/2.
  \end{equation}

Hence, for $m \ge 0$ and $k \geq 1$, we obtain
\begin{align*}
{\mathbf P} (X_{k+m+1}^x\in A) 
&\ge 
\sum_{i=0}^m  
{\mathbf P} (\chi_k^x = i, X_{k+m+1}^x \in A) \\
&= \sum_{i=0}^m  {\mathbf P} (X_{k+m+1}^x \in A| \chi_k^x = i  ) {\mathbf P}(\chi_k^x=i) \\
  &\geq \sum_{i=0}^m \varepsilon^{m-i}\varepsilon {\mathbf P}(\chi_k^x=i)\\
  &\geq \varepsilon^{m+1} {\mathbf P}(\chi_k^x\leq m) 
\end{align*}
which, by \eqref{choose},  is bounded from below by $\frac 12 \varepsilon^{m_0+1}$ in case $m=m_0$ and $k\geq k_0$. The same lower bound holds for ${\mathbf P} (X_{k+m+1}^y \in B)$.

\smallskip

To complete the proof of the lemma in case $N=1$, we proceed similarly to Step 2 in the proof of the previous lemma. Consider $x,y \in {\cal E}$ and take an {\em independent} coupling of the chains $X^x$ and $X^y$. We showed that
$$
{\mathbf P} (X_{n}^x \in A, X_{n}^y \in B)\geq \frac 14 \varepsilon^{2m_0+2}
$$
for all sufficiently large $n$ (depending on $x$ and $y$). Choosing $T(x,y)$ such that ${\mathbf P} (X_{n}^x \in A, X_{n}^y \in B)\geq \frac 18 \varepsilon^{2m_0+2}$ for $n \geq T(x,y)$ and then, recursively,
$T_1=T(x,y)$, $T_i=T_{i-1}+T\big(X^x_{T_{i-1}},X^y_{T_{i-1}}\big)$, $i=2,3,...$, we see that $A \times B$ is achievable.

\smallskip

Finally, we treat the case $N \geq 2$. We define $Y_n^x:=X_{nN}^x$, $n \in \N_0$. The assumptions in the lemma (for $X$) imply that the assumptions also hold for $Y$ with $N=1$ and therefore,
$A \times B$ is achievable for $Y$. Since, as we showed above, $A \times B$ is achievable via independent coupling, the same holds for $X$ and the proof is complete.
\end{proof}

 \section{Conditions for existence of a stationary regime}
 
 \begin{prp} \label{L2}
Suppose that condition (BMC) holds. 
Suppose in addition that the following condition holds:\\

 \noindent {\bf (condition (PR))} There are two points $x_0\preceq y_0$ such that
 the hitting times
 \begin{align*}
 \nu^- \equiv \nu^{-}_{x_0} = \min \{n\ge 1: X_n^{x_0} \succeq x_0\} \quad  \text{and}
 \quad
 \nu^+ \equiv\nu^{+}_{y_0} = \min \{n\ge 1:  X_n^{y_0} \preceq  y_0\}
 \end{align*}
 have finite expectations.\\
 
 Then the Markov chain admits at least one stationary distribution.
 \end{prp}

\begin{proof}[Proof of Proposition \ref{L2}.]

We introduce two auxiliary Markov chains, the "lower" and the "upper". 
 Start with $L_0:=X_0^{x_0}=x_0$ and let $\nu^-$ be
 the first time $n\ge 1$ when $X_n^{x_0} \succeq x_0$ and define 
 $L_n=X_n^{x_0}$  for $n<\nu^-$. Let  $L_{\nu^-}=x_0$ and assume that the process follows the transition probabilities of the original chain after time $\nu^-$ until it hits the set
 $\{x \succeq x_0\}$ next time after time $\nu^-$, say, at time $\nu^-_2$, and let $L_{\nu^-_2}=x_0$. Continue in this way by induction.    
Then we obtain a regenerative Markov chain $(L_n)_{n\ge 0}$ with an atom at $x_0$ and with a finite mean return time ${\mathbf E} \nu^-$ to $x_0$. (There will be no need to
 specify the transition probabilities of this chain starting from states which cannot be reached from $x_0$).

Define the probability distribution 
\begin{align*}
\pi^- (\cdot ) = \frac{{\mathbf E} \left(\sum_{n=0}^{\nu^--1} \1_{\{X_n^{x_0} \in \ \cdot \ \}} \right)}{{\mathbf E} \nu^-}
\end{align*}
and let  $\{\widehat{L}_n\}$, $n \in \N_0$ be the associated stationary chain, i.e.~we start at time 0 with a $\pi^-$-distributed random variable $\widehat{L}_0$
and follow the dynamics to the chain $(L_n)$.

From the classical theory of regenerative Markov chains, 
it is known that $\{\widehat{L}_n\}$ {\em is} a stationary Markov chain and that, for any  measurable set $B\in \E$,
\begin{align}\label{piplus}
\lim_{n\to\infty} \frac{1}{n} \sum_{1}^n \1_{\{L_n\in B\}} 
=
\lim_{n\to\infty} \frac{1}{n} \sum_{1}^n {\mathbf P}(L_n\in B) 
=  \pi^-(B) \quad \text{a.s.}
\end{align}

Analogously, we introduce an "upper" chain $(U_n)_{n \in \N_0}$ that starts from
state $U_0=y_0$ and which moves as $X_n^{y_0}$ until it hits the set $\{y \preceq y_0\}$ at time $\nu_+$. Then let $U_{\nu^+}=y_0$ and restart from $y_0$. Then use induction to continue the construction, with restarting from $y_0$ each time when hitting the
set  $\{y \preceq y_0\}$.  Then define its stationary version  $\hat U_n$, $n \geq 0$ with stationary measure  
\begin{align*}
\pi^+ (\cdot ) = \frac{{\mathbf E} \left(\sum_{n=0}^{\nu^+-1} \1_{\{X_n^{y_0} \in \ \cdot \ \}} \right)}{{\mathbf E} \nu^+}.
\end{align*}
As above, we conclude that
\begin{align}\label{piminus}
\lim_{n\to\infty} \frac{1}{n} \sum_{1}^n \1_{\{U_n\in B\}}  
=
\lim_{n\to\infty} \frac{1}{n} \sum_{1}^n {\mathbf P}(U_n\in B) 
=  \pi^+(B) \quad \text{a.s.}
\end{align}

Recall that $x_0\preceq y_0$ and, due to monotonicity of the kernel of the original Markov chain, we get that ${\mathbf P} (L_n\in D) \le {\mathbf P} (U_n\in D)$, for any $n \geq 0$ and any set $D\in \D$.
Based on the second equalities in \eqref{piplus} and \eqref{piminus}, we may conclude that $\pi^- \preceq \pi^+$.

Next, we consider the sequence of probability measures $\mu_0:=\pi^-$, $\mu_{n+1}:=\mu_nP=\pi^- P^{(n)}$, $n \in \N_0$. From the definition of $(L_n)$ we see that
$\pi^- P \succeq \pi^-$ and since $P$ is monotone, it follows that the sequence $\mu_n$, $n \ge 0$ is increasing. We claim that the sequence is bounded by $\pi^+$. To see this, 
recall that $\pi^- \preceq \pi^+$ and, by monotonicity of $P$, $\pi^+ \succeq \pi^+ P^{(n)} \succeq  \pi^- P^{(n)}=\mu_n$. By the (BMC) property and Corollary
\ref{BMCmeasures}, we see that the sequence
$\mu_n$, $n \ge 0$ has a weak limit $\mu=\mu_{\infty}$. Since we did not assume $P$ to enjoy the Feller property we cannot conclude that the limit is invariant under $P$ but we can at least conclude that $\mu P\succeq \mu$ and, moreover,
$\mu P^{(n+1)} \succeq \mu P^{(n)}$ for all $n\ge 0$. We will now apply Zorn's lemma to show the existence of a stationary measure.

Recall that the space  ${\cal M}_1({\cal E})$ equipped with the topology of weak convergence and the stochastic order is an ordered Polish space (Proposition \ref{orderedPolish}). Let ${\cal M}$ be the subset of probability measures $\mu$ on $\cal E$ for which all $\mu P^{(n)}$, $n \in \N_0$ are bounded above by $\pi^+$ and which are superinvariant, i.e.~$\mu P^{(n+1)} \succeq \mu P^{(n)}$, for all $n\ge 0$. Note that $\cal M$ contains $\pi^-$ and that $\mu \in {\cal M}$ implies $\mu P \in \cal M$.
In order to apply  Zorn's lemma, we need to verify that for each
totally ordered subset $\cal T$ of $\cal M$ there exists some upper bound in $\cal M$. 

Assume this is not the case and that there exists a totally ordered subset $\cal T$ that does not have an upper bound.
As a subspace of the separable metric space ${\cal M}_1({\cal E})$, the space $\cal T$ is itself separable, so  we can find a countable dense subset $\cal A$ of $\cal T$. If $\cal T$ has a maximal element,
then this element is an upper bound and we are done, so assume that $\cal T$ does not possess a maximal element. Let $t \in \cal T$. We claim that there exists some $a \in \cal A$
which satisfies $a \succeq t$.     Since $t$ is not maximal there exists some $u\in \cal T$ which satisfies $u \succ t$. Let $a_n$ be a sequence in $\cal A$ which converges to $u$.
If any of the $a_n$ is greater or equal to $t$ then it is an upper bound of $t$. Otherwise all of the $a_n$ are smaller than $t$. Since  ${\cal M}_1({\cal E})$ is an ordered Polish space, the limit $u$ of the sequence satisfies $u \preceq   t$  which is a contradiction.  Next, we take any strictly increasing sequence $\bar a_n$ of elements in $\cal A$ with the property that for each $b \in \cal A$ there is
some $a$ in the sequence so that $a \succeq b$. This sequence has the property that for each $t \in \cal T$ there is an element in the sequence which dominates $t$. Since 
$\bar a_n \preceq \pi^+$, (BMC) implies that the sequence converges weakly to some $\bar a$.  Then, $\bar a \in \cal M$.

Now, Zorn's lemma shows that there is at least one maximal element $\mu_{\infty}$ in $\cal M$.  In particular, $\mu_\infty$ is superinvariant. Assume that  $\mu_\infty$ in not invariant. Then
$\mu_\infty P^{(2)} \succeq \mu_\infty P \succ \mu_\infty$, so $\mu_\infty P \in {\cal M}$ is strictly larger than $\mu_\infty$ contradicting our assumption that $\mu_\infty$ is not invariant, i.e.~$\mu_\infty$ {\em is} invariant and the proposition is proved.
\end{proof}

\begin{rem}\label{remex}
  Proposition \ref{L2} does not hold without the (BMC) condition. As an example, let $\E=\{\pm \frac 1n,\,n\in\N\}$ equipped with the discrete metric and   the order induced by that of the real line. Consider the order preserving Markov kernel $P\Big(\frac 1n, \big\{\frac 1{n+1}\big\}\Big)=P\Big(-\frac 1n,\big\{ -\frac 1{n+1}\big\}\Big)=1$, $n \in\N$. All conditions of Lemma \ref{L2} hold except for  the (BMC) condition. Obviously, there is no stationary distribution.
  \end{rem}

\begin{rem}
The conditions of Proposition \ref{L2} do not guarantee uniqueness of a stationary distribution. As an example, let $\E=[0,1]$ with the natural order and metric. Let $P$ be the identity, i.e.~$P(x,\{x\})=1$ for every $x \in [0,1]$. All conditions of Proposition \ref{L2} hold and there are infinitely many stationary distributions.
\end{rem}

\section{Conditions for uniqueness of the stationary regime}

Consider again a  monotone Markov chain $(X_n)_{n \in \N_0}$ on an ordered Polish state space ${\cal E}$. 
The following theorem establishes uniqueness of a stationary distribution under different conditions than Theorem \ref{L1}. Here, we do not assume that $M$ is achievable.

\begin{thm}\label{un1}
 Suppose that the following condition holds:\\
 For any $x,y\in \E$, there exist two measurable sets $A_{x,y}$ and $B_{x,y}$ such that $A_{x,y} \preceq B_{x,y}$ and the following
 two random variables are finite a.s.:
 \begin{align}\label{fin}
 \tau^x_{x,y} = \inf \{n\ge 0: \ X_n^x \in A_{x,y}\} \ \ \text{and}
 \ \ \tau^y_{x,y} = \inf \{n\ge 0: X_n^y \in B_{x,y}\}.
 \end{align}
 Then the Markov chain  possesses at most one stationary distribution.
\end{thm}

\begin{rem}
For any $x,y\in \E$, if
\begin{align}
\widehat{\tau} := \inf \{n\ge 0: X_n^y \succeq x \} <\infty \ \ \text{a.s.},
\end{align}
then condition \eqref{fin} follows. Indeed, it is enough to take $A_{x,y} = \{z\in\E : z\preceq x\}$ and $B_{x,y} = 
\{z\in\E : z \succeq x\}$. Then 
$\tau^x_{x,y} \equiv 0$ and $\tau^y_{x,y} = \widehat{\tau}$,
\end{rem}

\begin{rem}
The finiteness of random variables in \eqref{fin} is equivalent to the following condition: there exists a coupling of two Markov chains
$X_n^x$ and $X_n^y$ such that
\begin{align}\label{fin2}
{\mathbf P} (\tau^x_{x,y}<\infty, \tau^y_{x,y}<\infty) =1.
\end{align}
Indeed, condition \eqref{fin2} implies a.s. finiteness of each of the
two random variables. On the other hand, given \eqref{fin}, one can take the independent coupling of the two Markov chains to
arrive at \eqref{fin2}.
\end{rem}

\begin{rem}
  The assumptions of Theorem \ref{un1} do not imply compressibility and if a stationary measure $\pi$ exists, then weak convergence of all transition probabilities does {\em not} necessarily hold.   
   As an example, take $\E=\{0,1\}$ with the trivial order and the chain which jumps from 0 to 1 and from 1 to 0 with probability one. The chain is monotone and $A=B=\{0\}$ is recurrent. Clearly, \eqref{conv} fails, for example, for $g=\1_{\{0\}}$ and $x\neq y$. The uniform distribution $\pi$  on ${\cal E}$ is the unique stationary measure but there is no weak  convergence of transition probabilities. Note that in this example the set $M$ is not achievable.
 \end{rem}
 
 \begin{proof}[Proof of Theorem \ref{un1}] 
   Assume that $\pi_1$ and $\pi_2$ are two distinct stationary measures. Then there exist two distinct {\em ergodic} stationary measures  (see, e.g., \cite{Hai}).
   Hence we can and will assume that $\pi_1$ and $\pi_2$ are distinct and ergodic. Fix an arbitrary set $I \in \D$. By Birkhoff-Khinchin's  ergodic theorem, the set $S_I^j$ of initial conditions $x\in \E$
   for which $\frac 1n \sum_{i=0}^n \1_{X_i^x\in I}$ converges to $\pi_j(I)$ satisfies $\pi_j(S_I^j)=1$, $j\in \{1,2\}$. Fix some $x\in S_I^1$ and $y \in S_I^2$. Let 
   $\tau_1$  be the first time when the chain $X:=X^x$ hits the set $A_{x,y}$ and  let $\tau_2$  be the first time when the chain $Y:=X^y$ hits the set $B_{x,y}$.
   Couple the chains  such that chain $X$ until time $\tau_1$ and the chain $Y$ until time $\tau_2$ are independent. Afterwards, couple the chains such that
   $X_{\tau_1+n}\preceq Y_{\tau_2+n}$ almost surely for all $n \in \N$. This, together with the argument above, implies $\pi_1(I)\leq \pi_2(I)$.  Interchanging $x$ and $y$ in the arguments above leads to the conclusion that 
   $\pi_1(I)=\pi_2(I)$ for every $I \in \D$. Hence $\pi_1=\pi_2$ by Lemma \ref{cons}  contradicting our assumption.
\end{proof}

\section{Stochastic stability}

In this section, we formulate criteria for the 
stability of a monotone Markov chain by combining results from previous sections. By {\it stability}, we mean the ability of the
Markov chain to {\it stabilise in time},  namely, the existence
of a unique stationary distribution $\pi$ and convergence of the distribution of $X_n^x$ to $\pi$ (in some sense) as $n\to \infty$, for any initial value $x\in\E$.
As a corollary (formulated in Remark \ref{remBM}), we obtain a version of a  basic result by Bhattacharya and Majumdar \cite[Theorem 5.1 in Section 3.5]{BM}, see also \cite{BW}. 

\begin{thm}\label{prop6}
Assume that conditions (BMC) and (PR) hold and that the set $M$ is achievable. Then there exists a unique stationary measure $\pi$ and  the statements \eqref{conv} and (ii) to (iv) of Theorem \ref{L1}
hold. 

If, in addition, for $\tau_{x,y}(M)$, $x,y \in \cal E$ as in \eqref{conv}, we have
\begin{equation}\label{beta}
\sup_{x,y \in \cal E}{\mathbf{P}}\big(\tau_{x,y}(M)>n\big)\leq \beta_n,
\end{equation}
then
\begin{equation}\label{betag}
\sup_{g \in \G} \big| {\mathbf{E}} g(X_n^x)-\int g(y)\, \d \pi(y)\big| \leq \beta_n
\end{equation}
and, in particular,
$$
\sup_{I \in \D} | {\mathbf P} (X_n^x \in I) - \pi (I)|\leq \beta_n.
$$
\end{thm}

\begin{proof}
The first claim is an immediate consequence of Proposition \ref{L2} and Theorem \ref{L1}.  Further, \eqref{betag} follows from \eqref{conv}, \eqref{convint}, and \eqref{beta}.
  \end{proof}

\begin{cor} \label{prp1}
Assume that condition (BMC) holds. Assume further that the following splitting condition holds: there exist $c\in \E$, $\varepsilon >0$ and $N\in \N$ 
such that, for $A_c = \{x\in \E : x\preceq c\}$ and 
$B_c = \{x\in\E: x\succeq c\}$, and for any $z\in \E$, 
\begin{align}\label{BM1}
P^{(N)} (z, A_c)\equiv {\mathbf P} (X_N^z\in A_c)\geq \varepsilon  \quad \text{and} \quad 
P^{(N)} (z, B_c) \geq \varepsilon.
\end{align}
Then the Markov chain admits a unique stationary distribution $\pi$, and,
for any $x\in \E$,  $n \in \N$,
\begin{align}\label{BM2}
\sup_{I \in \D} | {\mathbf P} (X_n^x \in I) - \pi (I)|\le (1-\varepsilon)^{\lfloor \frac nN \rfloor} .
\end{align}
In particular, there is geometric convergence of the transition probabilities to the stationary distribution in the uniform (Kolmogorov) metric.
\end{cor}

\begin{proof}[Proof of Corollary \ref{prp1}]
  Note that $M$ is achievable using an independent coupling and that (PR) holds with $x_0=y_0=c$. Using independent coupling, \eqref{beta} holds with
  $\beta_n=(1-\varepsilon^2)^{\lfloor \frac nN \rfloor}$,
  so \eqref{BM2} holds with $1-\varepsilon^2$ instead of $1-\varepsilon$. 

  To see that the convergence rate can be improved to $1-\varepsilon$, we first consider the case $N=1$. We can find a coupling on 
  ${\cal E} \times {\cal E}$  such that $P\big( (x,y),A_c \times B_c\big)\geq \varepsilon$ for all $x,y \in \cal E$: for given $x\neq y$, toss a coin which comes up heads with probability $\varepsilon$ and in this case let $X^x_1 \in A_c$ and $X^y_1 \in B_c$ in such a way that both have the correct law. Then $M$ is achievable and    so  \eqref{BM2} holdsin case $N=1$.

  For $N \ge 2$ we argue as follows. Considering the chain $X^x_{kN}$ for $k=0,1,...$ we just saw that
  $$
|{\mathbf P}(X^x_{kN}\in D)-\pi(D)|\leq (1-\varepsilon)^{k},\;k \in \N_0,\;x \in {\cal E},\; D \in \D. 
$$
For $m \in \{0,...,N-1\}$, $x \in \cal E$, denote the law of $X_m^x$ by $\nu_{x,m}$. Using the Markov property we obtain, for $D \in \D$,
\begin{align*}
  \big|{\mathbf P}(X^x_{kN+m}\in D)-\pi(D)\big|&=\big|{\mathbf E}\big[ {\mathbf P}(X_{kN}^{X_m^x}\in D)|X^x_1,...,X^x_m)\big]-\pi(D)\big|\\
                                               &=\big|{\mathbf E}\big[ {\mathbf P}(X_{kN}^{X_m^x}\in D)|X^x_m)\big]-\pi(D)\big|\\
                                               &=\big|\int_{\mathcal E}   {\mathbf P}(X_{kN}^y\in D)-\pi(D)\,\d \nu_{x,m}(y)   \big|\\
                                               &\leq \int_{\mathcal E} \big|{\mathbf P}(X_{kN}^y\in D)-\pi(D)\big|\,\d \nu_{x,m}(y)  \\
                                               &\leq (1-\varepsilon)^{k},
\end{align*}
so \eqref{BM2} holds for general $N \in \N$.
\end{proof}

\begin{rem}\label{remBM}
  The set-up of Corollary \ref{prp1} contains the case in which $\cal E$ is a $G_\delta$ subset of $\mathbb{R}^k$ for some $k \in \N$ (i.e.~$\cal E$ is a countable intersection of open sets in $\mathbb{R}^k$). Such a set is a Polish space with respect to the trace topology (see \cite[Proposition 8.1.5]{C}). If, in addition, $\preceq$ is a partial order on $\cal E$ (e.g.~the componentwise order) and (BMC) holds, then Corollary \ref{prp1} can be applied. Particular cases of $G_\delta$ sets in $\mathbb{R}^k$ equipped with the componentwise order which satisfy (BMC) are closed subsets of $\mathbb{R}^k$.

  Results of this type were initially obtained in Dubins and Freedman \cite{DF} under the additional assumption of continuity of the transition kernel. The original proofs in \cite{BM,BM2} are different from ours. 
  \end{rem}

\section{Discussion}
\subsection{A Markov chain as a stochastic recursion}

{\it Coupling representation of a Markov chain.} 
It is known (and follows, say, from \cite{BF}, Section 2) that any Markov chain  $(X_n)_{n \in \N_0}$ 
 on a Polish space $({\cal E, \EE})$ may be represented as 
\begin{align}\label{SR0}
X_{n+1}=f(X_n,\eta_{n}), \ \text{for all} \ \ n,
\end{align}
where the {\it driving sequence} $\{\eta_n\}$ is i.i.d. and may be taken real-valued with the uniform $U(0,1)$ distribution, and $f: {\cal E} \times (0,1) \to {\cal E}$
a measurable function on the product space.
We will say that \eqref{SR0} is a {\it stochastic recursion}, or
an {\it SR (coupling) representation} of a Markov chain.

{\it SR representation of two Markov chains.}  
One can couple the chains  $(X_n^x)$ and $(X_n^y)$ with initial states $x\neq y$ on a common probability space using recursions $X_{n+1}^x=f(X_n^x,\eta_{n}^x)$ and
$X_{n+1}^y=f(X_n^y,\eta_{n}^y)$, where each of the driving sequences $\{\eta_n^x\}$ and $\{\eta_n^y\}$ is i.i.d., but their joint distribution at any time $n$ may be arbitrary and, in particular, may depend on $n$ and on the history and on the future of the sequence. 

The following result holds.
\begin{cor}\label{cor1}
Under the conditions of Proposition \ref{prosi}, there exists a coupling representation of two Markov chains, $(X_n^x)$ and $(X_n^y)$ as stochastic recursions \eqref{SR0} such that
$X_n^x\preceq X_n^y$ a.s., for all $n$. Moreover, thanks to Proposition 4 from \cite{KKO}, a similar coupling construction is valid for any linearly
ordered sequence of initial states $x_1\preceq x_2 \preceq \ldots \preceq x_n \preceq \ldots$.
\end{cor}

This allows us to provide alternative formulations and proofs
of our results using a parallel ``language'' of stochastic 
recursions. There are pro's and con's for doing that. We decided to follow a more standard terminology,
with stochastic kernels.

\begin{rem}
One might expect that if a Markov chain is monotone on a partially ordered state space, then there always exists a coupling representation \eqref{SR0} with function $f$  {\it monotone in the first argument}.  However, this is not true, in general, as was shown in \cite[Example 1.1]{FM} with a finite state space ${\cal E}$ containing only four elements, $|\E |=4$. On the other hand, this fact is correct in the case of linear (complete) ordering (see again \cite{FM}).
\end{rem} 

\subsection{Splitting conditions and stability}

In this subsection, we provide several examples that may help to clarify  relations between the splitting conditions and stability.

It was pointed out in \cite{BM}  that the splitting conditions are close to necessary for stability in the case of finite-dimensional Euclidean spaces. Here we provide two examples that show that this is not always the case. Further, we provide an example  that shows that, in general, under the assumptions of Theorem 6.1, we cannot expect 
convergence in a stronger sense than in the uniform metric. 

\begin{exa}\label{eq:exS1}
Let ${\cal E} = [0,1]$ and $(U_n)_{n\ge 0}$ an i.i.d.~sequence of uniform on $[0,1]$ random variables, and let 
\begin{align*}
X_{n+1}^x = X_n^x + \frac{U_n}{2} \left(
\frac{1}{2}-X_n^x\right), \  \text{for} \  X_0^x=x\in [0,1] \ \text{and} \  n=0,1,\ldots .
\end{align*}
Here  $(X_n^x)$ is stochastically monotone and weakly converges to its 
unique stationary distribution that is degenerate at point $1/2$. However, if $x\neq 1/2$, then $X_n^x \neq 1/2$ for all $n$ and there is no convergence in the uniform metric and, therefore, there is no convergence  in total variation. The splitting conditions  of Theorem 6.1 and Corollary \ref{prp1} do not hold for this Markov chain.
\end{exa}

Assume now that ${\cal E}$ is a closed subset of ${\mathbb{R}}^k$ with $k>1$. Then a Markov chain may converge to the unique limiting distribution 
in a strong sense (say, in total variation), while the splitting conditions from Theorem 6.1 do not hold.
We provide a very simple example.
\begin{exa}\label{eq:exS2}
Consider the interval 
\begin{align*}
{\cal E} = \{(x_1,x_2): \  x_1\ge 0, x_2\ge 0, x_1+x_2 = 1/\sqrt{2} \}.
\end{align*}
Assume that we use the standard component-wise partial ordering: $x=(x_1,x_2) \preceq y=(y_1,y_2)$ iff 
$x_1\leq y_1$ and $x_2 \leq y_2$.
Then the partial ordering in ${\cal E}$ is trivial: any point of ${\cal E}$ is comparable with itself only.
Assume that $P(x,A) = l(A)$ for any measurable set $A\in {\cal E}$ where $l(A)$ is its Lebesgue measure. Then the Markov chain converges to its stationary (uniform) distribution, but there is no splitting.

\end{exa}	

We end this subsection with one more example. In Theorem 6.1,
we show the convergence to the stationary distribution in the uniform metric $\sup_{I \in {\cal I}} |{\P} (X_n^x \in I) - \pi (I)|$.
However, the stronger convergence in the total variation may not follow, in general. 

\begin{exa}\label{eq:exS3}
Let ${\cal E}=[0,1]$ and consider a Markov chain 
$$
X_{n+1} = \frac{1}{2}(X_n+\xi_{n+1}), \ \ n=0,1,\ldots,
$$
where
$\{\xi_n\}_{n\ge 1}$ is an i.i.d. sequence of Bernoulli random variables. 
Here the splitting condition from Theorem 6.1 clearly holds with $c=1/2$, i.e. $\P (X_1^x \ge 1/2) \ge 1/2$ and $\P(X_1^x \le 1/2) \ge 1/2$, for all $x\in [0,1]$. The limiting distribution is uniform on $[0,1]$, it is (absolutely) continuous w.r.t. Lebesgue measure.
However, for any $n$ and any fixed $x$, the distribution of $X_n^x$ is discrete. Therefore, there is no convergence in the total variation.
\end{exa} 
Similar conclusions may be made for many other discrete distributions of $\{\xi_n\}$ on $[0,1]$. However, if the distribution of $\{\xi_n\}$ has an absolutely continuous component, then the convergence  in the total variation follows.

\section{Generalisations}

One can expect the results of this paper to be generalised to non-Markovian settings and, in particular, to Markov-modulated (Markov-adapted, in another terminology) Markov chains and
to stochastic recursions with stationary ergodic driving sequences, see e.g. \cite{FSTW} for terminology. For that, one may try to combine arguments from the current paper and from \cite{FSTW}
where the case of a bounded state space was considered. However, such generalisations are not straightforward and require further work.


\appendix
\section{Appendix: Proof of Proposition \ref{prp3.2} }


\begin{proof}[Proof of Proposition \ref{prp3.2}]
We will prove that
$(\widetilde{X}_n,\widetilde{Y}_n)_{n\ge 0}$ is a marginally Markovian $(x,y)$-coupling with transition kernel $P$. Note that  $\tau$ is also a stopping time with respect to the filtration
$\mathcal{G}_n:=\sigma(\widetilde{X}_k,\widetilde{Y}_k)_{k \leq n}, n\in \N_0$. For fixed $n\in \N_0$ and $A \in \EE$, we will show that
$$
{\mathbf P}\big(\widetilde {X}_{n+1}\in A\,|\,\mathcal{G}_n\big)=P\big(\widetilde X_n, A\big),\mbox{ a.s.}
$$

The right hand side is $\mathcal{G}_n$-measurable, so we have to show that, for each event $G \in \mathcal{G}_n$, 
\begin{align}\label{identi}
{\mathbf E} \left(\1_{\widetilde{X}_{n+1}\in A} \1_G\right) = {\mathbf E} \left(
P(\widetilde X_n, A) \1_G\right) 
\end{align}
(the corresponding equality for $\widetilde{Y}$ will follow in the same way).  It is enough to show \eqref{identi} for $G$ replaced by $G\cap\{\tau \geq n+1\}$ and for $G$ replaced by $G\cap\{\tau=k\}$ for every $k \leq n$.
Note that all of these events are in $\mathcal{G}_n$ since $\tau$ is a stopping time with respect to  $\mathcal{G}_n,\,n \in \N_0$. We have
\begin{align*}
{\mathbf E} 
\left(
   \1_{\widetilde{X}_{n+1}\in A} 
  \1_{G\cap\{\tau \geq n+1\}}
  \right)
  &= 
  {\mathbf E} 
  \left(
 \1_{X^x_{n+1}\in A} 
  \1_{G\cap\{\tau \geq n+1\}}
  \right) 
  =
  {\mathbf E}
  \left(
  P(X^x_n,A)
  \1_{G\cap\{\tau \geq n+1\}}
  \right) 
  \\
 &=
 {\mathbf E}
 \left(
  P(\widetilde{X}_n,A)
  \1_{G\cap\{\tau \geq n+1\}}
  \right)
\end{align*}
where the second equality holds since $(X^x_n,X^y_n)_{n\geq 0}$ is a marginally Markovian $(x,y)$-coupling. For $k \in \{0,\cdots,n\}$ we have, by definition,
\begin{align*}
{\mathbf P}\big(\hat X_{n+1-k}^{X_k^x,X_k^y} \in A\,|\,\mathcal{G}_n)= P(\hat X_{n-k}^{X^x_k,X_k^y},A)  \   \mbox{ on the event } \{\tau=k\},
\end{align*}
thus 
\begin{align*}
{\mathbf E}
\left(
  \1_{\widetilde{X}_{n+1}\in A} 
  \1_{G\cap\{\tau =k\}}
  \right)
  &= 
  {\mathbf E}
  \left(
  \1_{ \hat X_{n+1-k}^{X_k^x,X_k^y}     \in A} 
  \1_{G\cap\{\tau =k\}}
  \right)
  = {\mathbf E}
  \left(
  P( \hat X_{n-k}^{X^x_k,X_k^y}   ,A)
  \1_{G\cap\{\tau =k\}}
  \right)
  \\
  &=
  {\mathbf E}
  \left(
  P(\widetilde{X}_n,A)
  \1_{G\cap\{\tau =k\}}
  \right),
\end{align*}
so the claim follows.
\end{proof}

{\bf Acknowledgements}. The authors are very thankful to
Jim Fill and Motoya Machida for discussions and valuable
comments related to various coupling constructions. 
The authors also thank the anonymous referees and the editor for many  helpful comments  and suggestions. 
Michael Scheutzow acknowledges financial support from the London Mathematical Society, grant No. 22203.

\end{document}